\documentclass[12pt]{article}
\usepackage{amssymb,amsmath,amsfonts,amsthm,amsxtra,setspace}
\usepackage[T1]{fontenc}
\usepackage{lmodern}

\newcommand{\mc}{\mathcal}

\newcommand{\noopsort}[1]{}

\newtheorem{theorem}{Theorem}[section]
\newtheorem{cor}[theorem]{Corollary}

\newsavebox{\Prfref}

\newsavebox{\prfref}

\newtheoremstyle{ref}
{\topsep}	
{\topsep}	
{\it}
{}
{}
{}
{ }
{\thmname{{\bfseries#1}}\thmnumber{ \textbf{#2\thmnote{\rm #3}\textbf .}}}

\theoremstyle{ref}
\newtheorem{lem}[theorem]{Lemma}
\newtheorem{thm}[theorem]{Theorem}

\newtheorem{defn}[theorem]{Definition}
\newtheorem{ex}[theorem]{Example}
\newtheorem{que}{Question}

\newtheoremstyle{nnref}
{\topsep}	
{\topsep}	
{}
{}
{}
{}
{ }
{\thmname{\textbf{#1}\thmnote{\textrm{ #3}}\textbf{.}}}

\theoremstyle{nnref}

\begin{document}
\title{On the definability of Menger spaces which are not $\sigma$-compact}
\author{Franklin D. Tall{$^1$}, Se{\c{c}}il Tokg\"oz{$^2$}}

\footnotetext[1]{Research supported by NSERC grant A-7354.\vspace*{2pt}}
\footnotetext[2]{ The second author is supported by T\"UB\.{I}TAK 2219 Grant.\vspace*{2pt}}
\date{\today}
\maketitle

\begin{abstract}
\noindent Hurewicz proved completely metrizable Menger spaces are $\sigma$-compact. We extend this to \v{C}ech-complete Menger spaces, and consistently, to projective metrizable Menger spaces. On the other hand, it  is consistent that there is a co-analytic Menger metrizable space that is not $\sigma$-compact.
\end{abstract}

\renewcommand{\thefootnote}{}
\footnote
{\parbox[1.8em]{\linewidth}{$2000$ Math.\ Subj.\ Class.\ Primary 03E15, 03E35, 03E60, 54A25, 54D20, 54H05; Secondary 03E45, 54D40.}}\\
\renewcommand{\thefootnote}{}
\footnote
{\parbox[1.8em]{\linewidth}{Keywords and phrases: Menger, $\sigma$-compact, projective set of reals, analytic, co-analytic, 
K-analytic, determinacy, p-space, s-space, $\Sigma$-space.}}

\maketitle

\section {When Menger spaces are $\sigma$-compact}
\makebox[0pt]{}\vspace{-2ex}
 
 The \emph{Menger} property is a strengthening of Lindel\"ofness which  plays an important role in the study of \emph{Selection Principles}.

\begin{defn}
A space $X$ is \textbf{Menger} if, given any sequence $\{\mc{U}_n\}_{\small{n < \omega}}$ of open covers of $X$, there exist finite subsets $\mc{V}_n$ of $\mc{U}_n$ such that $\bigcup_{n < \omega}\mc{V}_n$ covers $X$.
\end{defn}
\makebox[0pt]{}\vspace{-2ex}

Hurewicz \cite{Hur25} proved that completely metrizable Menger spaces are $\sigma$-compact and conjectured that indeed all Menger spaces are.  His conjecture was disproved in \cite{MiFr}.  Since then, easier, ``natural'' counterexamples have been constructed {\textemdash } see e.g. \cite{Tsa011}.
It was apparently not realized until now that Hurewicz' theorem was not limited to metrizable spaces. We shall assume all spaces considered are completely regular. We shall prove:

\begin{thm}\label{thm1}
\v{C}ech-complete Menger  spaces are $\sigma$-compact.
\end{thm}
\textit{p-spaces} were introduced by Arhangel'ski{\u\i}  \cite{Arh65}; an equivalent definition is due to Burke \cite{Bur70}:

\begin{defn} 
A  space $X$ is a \textbf{p-space} if there exists a sequence  $ \{ \mathcal{U}_ n \}_{n<\omega}$ of families of open subsets satisfying the following condition    if for each $n$, $x\in U_n \in \mathcal{U}_n $ then
\item i) $\bigcap\limits_{n<\omega} \overline{U_n}$ is compact,
\item ii) $\{ \bigcap\limits_{\small {i\leq n}} \overline{U_i}:n\in \omega \}$ is an outer network for the set $ \bigcap\limits_{\small{n<\omega}} \overline{U_n}$ i.e., every open set including  $\bigcap\limits_{\small{n<\omega}} \overline{U_n}$ includes some $\bigcap\limits_{\small{i\leq n}} \overline{U_i}$.
\end {defn} 
Recall that  \textbf{paracompact p-spaces} 
 are preimages of metrizable spaces under perfect mappings \cite{Arh65}. 
\begin{proof}
A \v{C}ech-complete Lindel\"of  space $X$ is a paracompact $p$-space and so is a perfect pre-image of a separable metrizable space $Y$.  \v{C}ech-completeness is preserved by perfect maps, and Mengerness by continuous maps, so if $X$ is Menger, $Y$ is Menger and \v{C}ech-complete, so by \cite{Hur25} is $\sigma$-compact.  But a perfect pre-image of a $\sigma$-compact space is $\sigma$-compact. 
\end{proof}

\begin{lem}[{~\cite[a more accessible reference is  22]{Hur27}}]\label{lem2}
The irrationals are not Menger.
\end{lem}

\begin{lem}[{~\cite[a more accessible reference is 17]{Hur28}}]\label{lem3}
If $Y \subseteq \mathbb{R}$ is analytic, then either $Y$ is $\sigma$-compact or $Y$ includes a closed copy of the space $\mathbb{P}$ of irrationals.
\end{lem}

It follows immediately that:

\begin{thm}[{~\cite{Hur25}}]\label{thm4}
Analytic Menger subsets of $\mathbb{R}$ are $\sigma$-compact.
\end{thm}

Under the Axiom of Projective Determinacy (PD), Lemma \ref{lem3} holds for projective sets of reals  {\textemdash} see  e.g. \cite{Tal011}.  Thus

\begin{thm}\label{thm5}
PD implies projective Menger subsets of $\mathbb{R}$ are $\sigma$-compact.
\end{thm}

The reader unfamiliar with determinacy should consult Section 27 of \cite{Kan94}  and Sections 20 and 38.B of \cite{Kec95}.\\

Given our success with Theorem \ref{thm1}, one wonders whether Theorems \ref{thm4} and \ref{thm5} can be extended to Menger spaces which are not metrizable. Theorem \ref{thm4} can.

\begin{defn}[{~\cite{Arh86}}]
A space is \textbf{analytic} if it is a continuous image of the space $\mathbb{P}$ of irrationals.
\end{defn}

Arhangel'ski\u{\i} proved:

\begin{thm}[{~\cite{Arh86}}]\label{lem7}
Analytic Menger spaces are $\sigma$-compact.
\end{thm}

\noindent \textbf{Warning}.  Arhangel'ski\u{\i} and some other authors call Menger spaces \emph{Hurewicz}. Our terminology is now standard.

\section {Some counterexamples}
\makebox[0pt]{}\vspace{-2ex}

Theorem \ref{thm1} does not obviously follow from Theorem \ref{lem7} since Menger \v{C}ech-complete spaces need not be analytic. For example, it is easy to see that a compact space of weight $\aleph_1$ is not  analytic.

In \cite{TsTa011} the first author incorrectly claimed that Arhangel'ski\u{\i} in \cite{Arh86} had proven that analytic spaces are perfect preimages of separable metrizable spaces. Arhangel'ski\u{\i}  did not claim this and it is not true. What Arhangel'ski\u{\i} did prove was that a non-$\sigma$-compact analytic space includes a closed copy of the space of irrationals, which suffices to prove the application in \cite{TsTa011} of Arhangel'ski{\v{\i}}'s supposed theorem. A counterexample to this false claim can be constructed by taking a closed discrete infinite subspace $F$ of the space of irrationals and collapsing it to a point. The resulting space is clearly analytic and hence has a countable network. However it does not have countable weight, because it is not first countable. Perfect pre-images of separable metrizable spaces are Lindel\"of $p$-spaces; in $p$-spaces, network weight $=$ weight, so this space is not a $p$-space.~~\qedsymbol

It is not clear what the right definition of ``projective" should be so as to attempt to generalize Theorem \ref{lem7}. A possible approach is to define the \emph{projective spaces} to be the result of closing the class of analytic spaces under remainders and continuous images, and hope that Theorem \ref{lem7} extends to this class, assuming PD. \\

Unfortunately we immediately run into trouble. It is not clear whether Menger remainders of analytic spaces are consistently $\sigma$-compact. See the next section for discussion. Continuous images of such Menger remainders need not be  in fact $\sigma$-compact {\textemdash} here is a counterexample:\\
\begin{ex}
 Okunev's space \cite{Arh000} is obtained by taking the Alexandrov duplicate $A$ of $\mathbb{P}$ (without loss of generality in $[0,1]$) and collapsing the non-discrete copy of $\mathbb{P}$ to a point. The remainder of $A$ in the duplicate of $[0,1]$ is analytic; to see this, note it is just the Alexandrov duplicate $D$  of the space of rationals in $[0,1]$. $D$ is a countable metrizable space; it may therefore be considered as an $F_\sigma$ subspace of $\mathbb{R}$, so is analytic. Thus  Okunev's space is the continuous image of a remainder of an analytic space. In \cite{BuTa012}, this space is shown to be Menger and not $\sigma$-compact.
 \end{ex}

There is a widely used generalization of analyticity to non-metrizable spaces {\textemdash} see  e.g. \cite{RoJa80}:

 \begin{defn}
A space is \textbf{$K$-analytic} if it is a continuous image of a Lindel\"of \v{C}ech-complete space.
\end{defn}
 Analytic spaces are $K$-analytic, and $K$-analytic metrizable spaces are analytic \cite{RoJa80}.

\begin{thm}
 K-analytic Menger spaces with closed sets $G_\delta$ are $\sigma$-compact.
\end{thm}

\begin{proof}
By 3.5.3 of \cite{RoJa80}, if such a space were not $\sigma$-compact, it would include a closed
set which maps perfectly onto $\mathbb{P}$. But  $\mathbb{P}$  is not Menger.
\end{proof}

One might conjecture that Theorem  \ref{thm4} could be extended to K-analytic spaces, but unfortunately Arhangel'ski\u{\i}  \cite{Arh000}  proved Okunev's space is  K-analytic.

PD is equiconsistent with and follows from mild large cardinal assumptions, and is regarded by set theorists as relatively harmless.  This result leads us to ask whether it is consistent that there is a projective Menger set of reals which is not $\sigma$-compact.  The answer is:

\begin{thm}[~\cite{MiFr}] 
$V=L$ implies there is a Menger non-$\sigma$-compact set of reals which is projective.
\end{thm}

As the set-theoretically knowledgeable reader might expect, one can simply use a $\mathbf{\Sigma}_{2}^1$ well-ordering of the reals to construct the desired example, being more careful in constructing the Menger non-$\sigma$-compact example in \cite{Tsa011}. But we will do better later.\hfill\qedsymbol 

In \cite{Tal013},   the first author  proved that if II has a winning strategy in the {\it{Menger game}} (defined in \cite{Tel84}), then the space is projectively $\sigma$-compact, and asked whether the converse is true.  It isn't:

\begin{ex}
Moore's $L$-space \cite{Mo006} is projectively countable \cite{Tal013} but II does not have a winning strategy in the Menger game.
\end{ex}

\begin{proof}
By \cite{Tel84}, for hereditarily Lindel\"of spaces, II winning the Menger game is equivalent to $\sigma$-compactness.  Moore's space is not $\sigma$-compact, else it would be productively Lindel\"of, which it isn't \cite{TsZd012}.
\end{proof}

\section {Co-analytic Spaces}
\begin{defn}
For a space $X$ and its compactification $bX$ the complement $bX \setminus X$ is called a \textbf{remainder} of $X$.
\end{defn}

\begin{defn}
A space is \textbf{co-analytic} if its Stone-\v{C}ech remainder  $\beta X\setminus X$ is analytic.                            
\end{defn}
In this section, we investigate whether co-analytic Menger spaces are $\sigma$-compact.
\begin{lem}\label{any}
Any remainder of a co-analytic space is analytic.
\end{lem}
\begin{proof}
Let $X$ be  a co-analytic space. Then the remainder $\beta X \setminus X$ of $X$ is analytic. Any compactification of  a space $X$ is the  image of   $\beta X$ under a unique continuous mapping f that keeps the space  $X$ pointwise fixed; furthermore   $f( \beta X \setminus X)=bX \setminus X$  \cite{Cee}. This implies that $X$ is co-analytic in $bX$. Thus any remainder  of $X$ is analytic.
\end{proof}

\begin{thm}
If there is a measurable cardinal, then co-analytic Menger sets of reals are $\sigma$-compact.
\end{thm}
\begin{proof}
The hypothesis implies co-analytic (also known as  $\prod ^1_1$-) determinacy  (\cite{Kan94}, 31.1) and hence the co-analytic case of  Theorem \ref{thm5}. Indeed, that co-analytic determinacy implies co-analytic Menger sets of reals are $\sigma$-compact was noticed earlier in \cite{MiFr}.\end{proof}

Theorem \ref{lem7} did not require analytic determinacy, so it is not obvious that something beyond ZFC is needed in order to conclude that Menger co-analytic sets of reals are $\sigma$-compact. But it is.

We  have:
\begin{thm}
Suppose $\omega_{1}=\omega_{1}^L$ and $ \mathfrak d > \aleph_{1}$. Then there is a co-analytic Menger set of reals which is not $\sigma$-compact.
\end{thm}
\begin{proof}
$\omega_{1}=\omega_{1}^L$ implies there is a co-analytic set of reals of size $\aleph_{1}$ (\cite{Kan94}, 13.12). By \cite{Hur27},  any set of reals of size $< \mathfrak d$ is Menger. Every 
$\sigma$-compact set of reals, has cardinality either countable or $2^{\aleph_{0}}$, but $ \mathfrak d > \aleph_{1}$ implies $2^{\aleph_{0} }> \aleph_{1}$, so $X$ cannot be 
$\sigma$-compact.
\end{proof}
It is easy to find a model of set theory satisfying these two hypotheses. For example, start with L and force Martin's Axiom plus $2^{\aleph_{0}}= \aleph_{2}$ by a countable chain condition iteration.

A question we have been unable to answer is whether it is consistent  that, modulo large cardinals, co-analytic Menger spaces are $\sigma$-compact. One would expect to reduce the problem to sets of reals and apply co-analytic determinacy, but we have so far been unsuccessful. We however have some partial answers.

\begin{lem}\label{lemco}
The perfect image of a co-analytic space is co-analytic.
\end{lem}

\begin{proof}
Let $Y$ be a perfect image of a co-analytic space $X$ under  a perfect map $f$. Then $f$ has a unique continuous  extension $F:\beta X\rightarrow \beta Y$ such that $F(\beta X\backslash  X)=\beta Y\backslash  Y$ \cite[Lemma 1.5]{HenIsb57} .
Since the continuous image of an analytic space is analytic, $Y$ is co-analytic .
\end{proof}

\begin{thm}\label{p-an}
Co-analytic determinacy implies that if $X$ is the  Menger remainder of an analytic p-space, then $X$ is $\sigma$-compact.
\end{thm}
\begin{proof}
It is known that any remainder of  a Lindel\"of p-space is a Lindel\"of p-space \cite{Arh005}. As in the proof of Theorem \ref{thm1}, it suffices to prove that $X$ maps perfectly onto a separable metrizable space $Y$. Since $X$ is co-analytic, by using Lemma \ref{lemco} we obtain that $Y$ is co-analytic. But then $Y$ is $\sigma$-compact, so  $X$ is $\sigma$-compact.\hfill\qedsymbol
\end{proof}

\begin{defn} 
$X$ is  { \textbf {projectively \v{C}ech-complete}} if every separable metrizable image of $X$ is \v{C}ech-complete.
\end{defn}
\begin{thm}\label{3.9}
 Every co-analytic projectively \v{C}ech-complete space is \v{C}ech-complete.
 \end{thm}

\begin{proof}
Let X be co-analytic and projectively \v{C}ech-complete. Let $Y = \beta X\setminus X$. If Y is $\sigma$-compact, we are done. If not, then by [2],  $Y$  includes a closed $F$ homeomorphic to $\mathbb{P}$. Following the proof of Theorem 1.5. in \cite{Ok93}, we can extend a homeomorphism mapping $F$ onto  $\mathbb{P}$ to a continuous  $f: Y\rightarrow [0,1]$. We can choose a continuous $g: Y\rightarrow [0,1]$
such that $F=g^{-1}(\{ 0 \})$ and define $h:Y \rightarrow [0,1]^2$ by $h(x)=( f(x),g(x))$. The set $\{ (a,b)\in h(Y): b=0\}$ is closed in $h(Y)$ and is homeomorphic to $\mathbb{P}$, so $h(Y)$ is not $\sigma$-compact. Extend $h$ to $ \hat{h}$ mapping $\beta Y$ into $[0,1]^2$. Then $h(Y)= \hat{h}(Y)$ is not $\sigma$-compact, since $\{ (a,b)\in h(Y): b=0\}$ is closed in $h(Y)$ and is homeomorphic to $\mathbb{P}$. But then $\hat{h}(X)$ is not \v{C}ech-complete.
\end{proof}
\begin{cor}
Every co-analytic projectively  \v{C}ech-complete Menger space is $\sigma$-compact.
\end{cor}
\begin{proof}
By Theorem \ref {3.9} and Theorem \ref{thm1}.
\end{proof}

\begin{cor}
Co-analytic determinacy implies that if $X$ is co-analytic Menger and has closed sets $G_\delta$, then $X$ is $\sigma$-compact. 
\end{cor}
\begin{proof}
  The remainder of $X$ is a continuous image of  a separable metrizable space. Our extra hypothesis then implies $X$ is a p-space \cite{Arh013}. \end{proof}

\begin{defn}[~\cite{Na69}]
A space $X$ is  a  \textbf {$\Sigma$-space} if it is a continuous image of a perfect pre-image of a metrizable space. In particular, a Lindel\"of \textbf {$\Sigma$-space} is a continuous image of a Lindel\"of p-space.
\end{defn} 

\begin{cor}\label{14}
Co-analytic determinacy implies that co-analytic Menger $\Sigma$-spaces are $\sigma$-compact. 
\end{cor}

\begin{proof}
Let $X$ be Menger $\Sigma$-space and $Y$ be an analytic remainder of $X$. It suffices to prove $Y$ is a p-space, for then we can apply Theorem \ref{p-an}. Since $Y$ is analytic, it has a countable network and hence is hereditarily Lindel\"of. By Proposition 2.6. of \cite{Arh013},  if $X$ is  a Lindel\"of $\Sigma$-space, then a hereditarily Lindel\"of remainder of $X$  is   a p-space.
\end{proof}

After this paper was more-or-less complete,  A.V. Arhangel'ski\u{\i} sent us a copy of his paper \cite{Arh013a} in response to Question 2 below.  In \cite{Arh013a}, he introduces  
 the class of \textit{s-spaces:}
 
 \begin{defn}[~\cite{Arh013a}]  A Tychonoff  space $X$ is called an \textbf{ s-space} if there exists a countable \textbf{ open source} for $X$ in some compactification $bX$ of $X$, i.e. a countable collection $\mathcal{S}$ of open subsets of $bX$ such that $X$ is a union of some family of intersections of non-empty subfamilies of  $\mathcal{S}$. 
 \end{defn} 
  
  Arhangel'ski\u{\i} proves that:

\begin{lem}
If X is a space with remainder Y, then $X$ is a Lindel\"of $\Sigma$-space if and only if $Y$ is an s-space.
\end{lem}

\begin{lem}
$X$ is a  Lindel\"of p-space if and only if it is a Lindel\"of $\Sigma$-space and an s-space.
\end{lem}

These can be used to give another proof of Corollary \ref{14}. Arhangel'ski\u{\i} asks in the paper whether first countable $\sigma$-spaces with weight $\leq 2^{\aleph_{0}}$ are s-spaces. He told us, however, that what he should have asked is:

\begin{que}
Are $\sigma$-spaces of countable type and $w\leq 2^{{\aleph_{0}}}$ s-spaces?
\end{que}

We could prove that co-analytic determinacy implies co-analytic Menger spaces are $\sigma$-compact, if we could affirmatively answer the following:

\begin{que}
Is every space of countable type with a countable network metrizable?
\end{que}

Recall a space is {\textit{of countable type}} if each compact subspace  is included in a compact subspace which has a countable base for the open sets including it. Now, the analytic remainder of a co-analytic Menger space is of countable type, because  its remainder is Lindel\"of \cite{HenIsb57}. Similarly, the co-analytic Menger space is itself of countable type. Now the analytic remainder has a countable network, so   if Question 2 has an affirmative answer, it is metrizable. But then it is a Lindel\"of p-space, so its co-analytic remainder is a Lindel\"of p-space as well \cite{Arh005}. But now we can proceed as in the proof of Theorem \ref{thm1}.\hfill\qedsymbol
\\
\indent As Arhangel'ski\u{\i} points out, Question 2 is a special case of  Question 1. Notice that compact subspaces of a space with a countable network have a countable base. It follows that a space of countable type with a countable network is actually first countable; indeed its compact subspaces have countable character, because they are included in compact metrizable subspaces which have countable character in the whole space. So let us ask:\\
\begin{que} Suppose a space has a countable network and each compact subspace of it has countable (outer) character. Is the space metrizable? 
\end{que}

By Michael \cite{M} it would suffice to show such a space has a countable {\it k-network}.\\

Note, incidentally, that co-analytic determinacy has large cardinal strength. See \cite{Kan94}, where it is pointed out (see page 444) that  it is equiconsistent with the existence of $0^ \#$ ( and hence with the existence of a measurable cardinal). \\

Note that Okunev's space is not co-analytic. To see this, let $V$ denote Okunev's space. Since  the remainder of $V$ in its Stone-\v{C}ech compactification is Borel but not Baire \cite{BuTa012}, $\beta V\setminus V$ cannot  have closed sets $G_\delta$, so the remainder is not analytic. Hence  by Lemma \ref{any},  $V$ is not co-analytic.
  
  As promised earlier, 
 \begin{thm}[~{\cite{MiFr}}] \label{main} 
 $V=L$ implies that there is a co-analytic Menger set of reals which is not $\sigma$-compact. 
\end{thm}
The crucial point here is 

\begin{lem}[~{\cite{MiFr}}] 
$V=L$ implies that there is a co-analytic scale.
\end{lem}

\begin{defn}
For $f,g \in \omega ^ \omega$, $f\leq^* g$ if $f(n)\leq g(n)$ for all but finitely many $n\in\omega$. A \textbf{scale} is a subset of $\omega^ {\omega}$ well-ordered by $\leq^*$ such that each $f\in \omega ^ \omega$  is $\leq^*$ some member of the scale.
\end{defn}

Bartoszy\'{n}ski and Tsaban \cite{BaTs006} proved:

\begin{lem} Let $S\subseteq\omega^ \omega$ be a scale. Then $S\cup [\omega]^{< \omega}$ is a Menger set  of reals which is not $\sigma$-compact. 
\end{lem}
Thus, to complete the proof of Theorem \ref{main}, it suffices to prove that the union of  a co-analytic set of reals  with a countable set of reals is co-analytic. Let these be $A,B$ respectively. Then $A\cup B={\mathbb{R}}\setminus ((\mathbb{R}\setminus A)\cap (\mathbb{R}\setminus B))$. Both $\mathbb{R}\setminus A$ and  $\mathbb{R}\setminus B$ are analytic by $14.4 $ of \cite{Kec95}; the intersection of two analytic set is analytic, so $A\cup B$ is co-analytic.\hfill\qedsymbol

One can in fact from $V = L$ construct a co-analytic Menger topological group which is not $\sigma$-compact  {\textemdash } see \cite{To}.
On the other hand, co-analytic determinacy implies one cannot {\textemdash } see \cite{T}.

We conclude by asking again,
\begin{que}
Is there a ZFC example of a co-analytic Menger space which is not $\sigma$-compact?
\end{que}

\noindent \textbf{Note:} After this paper was submitted, Y. Peng answered Questions 2 and  3 in the negative.

{\rm Franklin D. Tall, Department of Mathematics, University of Toronto, \\Toronto, Ontario M5S 2E4, CANADA}\\
{\it e-mail address:} {\rm tall@math.utoronto.ca}\\

{\rm Se{\c{c}}il Tokg\"oz, Department of Mathematics, Hacettepe University,\\ Beytepe, 06800, Ankara, TURKEY}\\
{\it e-mail address:} {\rm secilc@gmail.com}
\end{document}